\documentclass{article}
\usepackage[utf8]{inputenc}
\usepackage{authblk}
\usepackage{setspace}
\usepackage[margin=1in]{geometry}
\usepackage{graphicx}
\graphicspath{ {./figures/} }
\usepackage{subcaption}
\usepackage{amsmath}
\usepackage{lineno}
\usepackage{comment}
\usepackage{amsthm}
\newtheorem{theorem}{Theorem}
\usepackage{amsfonts} 

\usepackage{url}
\usepackage{hyperref}
\usepackage{xcolor}
\hypersetup{
    colorlinks,
    linkcolor={red!50!black},
    citecolor={blue!50!black},
    urlcolor={blue!80!black}
}
\usepackage[natbib=true,backend=biber,sorting=none,style=ieee, citestyle=numeric-comp]{biblatex}

\title{Stochastic Time-Optimal Control Studies for Additional Food provided Prey-Predator Systems involving Holling Type-IV Functional  Response}

\author[1]{D Bhanu Prakash}
\author[2]{D K K Vamsi}

\affil[1, 2]{ \ Department of Mathematics and Computer Science, Sri Sathya Sai Institute of Higher Learning, India.}
\affil[1]{First Author. Email: dbhanuprakash@sssihl.edu.in}
\affil[2]{Corresponding author. Email: dkkvamsi@sssihl.edu.in}

\date{}

\begin{document}

\maketitle

\begin{abstract} {{
\noindent We consider an additional food provided prey-predator model exhibiting Holling type IV functional response with combined continuous white noise and discontinuous L\'evy noise. We prove the existence and uniqueness of global positive solutions for the considered model. By considering the quality and quantity of additional food as control parameters, we formulate a time-optimal control problem. We obtain the condition for the existence of an optimal control. Furthermore, making use of the arrow condition of the sufficient stochastic maximum principle, we characterize the optimal quality of additional food and optimal quantity of additional food. Numerical results are given to illustrate the theoretical findings with applications in biological conservation and pest management.}}
\end{abstract}

{ \bf {keywords:} } Stochastic optimal control; Time-optimal control; Holling type IV response; Biological conservation; Pest management; Brownian motion; L\'evy noise



\section{Introduction} \label{Intro}

\qquad The complex natural ecosystems present around us kindled great attention of many ecologists and mathematicians to the mathematical modeling of ecological systems in the last few decades. The interaction among species in these ecosystems can be of several forms like competition, mutual interference, prey-predator interactions and so on. The very first ecological models are framed from the pioneering works of Alfred J. Lotka \cite{lotka1925elements} and Vito Volterra \cite{volterra1927variazioni} in 1925. Various complex models are framed and studied ever since. 

The basic component of prey-predator systems is a functional response, which is defined as the rate at which each predator captures prey \cite{kot2001elements}. A few examples of functional responses include Holling functional responses \cite{metz2014dynamics}, Beddington–DeAngelis functional responses \cite{geritz2012mechanistic}, Arditi-Ginzburg functional responses \cite{arditi1989coupling}, Hassell-Varley functional responses \cite{hassell1969new} and Crowley \& Martin functional response \cite{crowley1989functional}. In this work, we study the models exhibiting Holling type IV functional response. Some of the organisms that display Holling type IV functional response in nature are found in \cite{yano2012cooperative, mcclure2011defensive}. In this functional response, the predator's per capita rate of predation decreases at sufficiently high prey density, due to either prey interference or prey toxicity.

In recent decades, many pioneering dynamical modeling works \cite{varaprasadsiraddi, varaprasadsirtime, v1,v2,V3,ananthcmb} reveal that the provision of additional food to predators plays a vital role in controlling the dynamics of the system. Recently, the authors in \cite{v2} have studied additional food provided deterministic prey-predator systems involving Holling type IV functional response. Also, the authors in \cite{ananth2022achieving,ananthcmb} studied the optimal control problems of deterministic prey-predator systems involving Holling type IV functional response with the quality of additional food and the quantity of additional food as the control parameters respectively.

Often it is observed that the parameters in an ecosystem are effected by the environmental fluctuations \cite{may2019stability}. Therefore in recent years, many researchers have drawn their attention to stochastic models. Most stochastic prey-predator models are driven by the Brownian motion, which captures the continuous noise. Authors in \cite{Stochastic5, Stochastic4} studied the deterministic and stochastic dynamics of a modified Leslie-Gower prey-predator system with simplified Holling type IV functional response. However, the sudden changes in environment like toxic pollutants, floods, earthquakes and so on, cannot be captured by the Brownian motion as it is a continuous noise. Hence, addition of a discontinuous noise, like L\'evy noise, to the prey-predator system with Brownian motion makes the models more realistic. \cite{Poisson1} uses the stochastic averaging method to analyze the modified stochastic Lotka-Volterra models under combined Gaussian and Poisson noise. \cite{LevyJumps2} studies the dynamics and dynamics of a Stochastic One-Predator-Two-Prey time delay system with jumps.

To the best of our knowledge, there is no study of additional food provided stochastic prey-predator system with jumps. Secondly, the optimal control studies of Stochastic Differential Equations with Jumps (SDEJ) were not performed on prey-predator systems. Lastly, very few works involved Holling type-IV response which incorporates the most important group defence property. Motivated by these observations, in this work, we study the optimal control problems for additional food provided stochastic Holling type IV prey-predator systems under combined Gaussian and L\'evy noise.

The article is structured as follows: Section \ref{sec2} introduces the stochastic prey-predator model with Holling type-IV functional response and additional food with intra-specific competition among predators. The existence of global positive solution for this model is briefly discussed in Section \ref{sec3}. The time-optimal control problem is formulated and the optimal quality and quantity of additional food is characterized in Section \ref{secOC}. Section \ref{secNS} illustrates the key findings of the analysis through numerical simulations in the context of both biological conservation and pest management. Finally, Section \ref{secD} presents the discussions and conclusions.

\section{The Stochastic Model} \label{sec2}

Let $N$ and $P$ denote the biomass of prey and predator population densities respectively. In the absence of predator, the prey growth is modelled using logistic equation. Further, we assume that the prey species exhibit Holling type-IV functional response towards predators. We also assume that the predators are supplemented with an additional food of biomass A, which is uniformly distributed in the habitat. Incorporating these assumptions, the prey-predator dynamics with Holling type-IV functional response along with additional food for predators can be described as:

\begin{equation} \label{det}
\begin{split}
\frac{\mathrm{d} N(t)}{\mathrm{d} t} & = r N(t) \left( 1-\frac{N(t)}{K} \right)- \Bigg( \frac{c N(t)}{(A \eta \alpha + a)(b N^2(t) + 1) + N(t)}\Bigg) P(t) \\
\frac{\mathrm{d} P(t)}{\mathrm{d} t} & = e \Bigg( \frac{N(t) + \eta A (bN^2(t) + 1)}{(A \eta \alpha + a)(b N^2(t) + 1) + N(t)} \Bigg) P(t) - m_1 P(t)
\end{split}
\end{equation}

For a complete analysis of model (\ref{det}), the reader is referred to Vamsi et. al., \cite{V3}.

In addition, we also assume that the predators exhibit intra-specific competition. We capture this competition in similar lines with \cite{stochastic8, CompImp}. Accordingly, the system (\ref{det}) gets transformed to the following system.

\begin{equation} \label{det2}
\begin{split}
\frac{\mathrm{d} N(t)}{\mathrm{d} t} & = r N(t) \left( 1-\frac{N(t)}{K} \right)- \Bigg( \frac{c N(t)}{(A \eta \alpha + a)(b N^2(t) + 1) + N(t)}\Bigg) P(t) \\
\frac{\mathrm{d} P(t)}{\mathrm{d} t} & = e \Bigg( \frac{N(t) + \eta A (bN^2(t) + 1)}{(A \eta \alpha + a)(b N^2(t) + 1) + N(t)} \Bigg) P(t) - m_1 P(t) - \delta P(t)^2
\end{split}
\end{equation}

The biological descriptions of the various parameters involved in the systems  (\ref{det}) and  (\ref{det2}) are described in Table \ref{param_tab}.

\begin{table}[bht!]
    \centering
    \begin{tabular}{ccc}
        \hline
        Parameter & Definition & Dimension \\  
        \hline
        T & Time & time\\ 
        N & Prey density & biomass \\
        P & Predator density & biomass \\
        r & Prey intrinsic growth rate & time$^{-1}$ \\
        K & Prey carrying capacity & biomass \\
        c & Maximum rate of predation & time$^{-1}$ \\
        e & Maximum growth rate of predator & time$^{-1}$ \\
        m$_1$ & Predator mortality rate & time$^{-1}$ \\
        $\delta$ & Death rate of predators due to intra-specific competition & biomass$^{-1}$ time$^{-1}$ \\ 
        A & Quantity of additional food for predators & biomass \\
        b & Group defence in prey & biomass$^{-2}$ \\
        \hline
    \end{tabular}
    \caption{Description of variables and parameters present in the systems (\ref{det}), (\ref{det2})}
    \label{param_tab}
\end{table}

\newpage

In order to reduce the complexity of the model, we non-dimensionalize the system (\ref{det2}) using the following non-dimensional parameters. 
$$N=ax, \  P=\frac{ay}{c}, \ \gamma = \frac{K}{a}, \ \xi = \frac{\eta A}{a}, \ \omega = b a^2, \  m_2 = \frac{c}{a\delta}$$.
Accordingly, system (\ref{det2}) gets reduced to the following system.

\begin{equation} \label{det3}
\begin{split}
\frac{\mathrm{d} x}{\mathrm{d} t} & = rx \Bigg(1-\frac{x}{\gamma} \Bigg)- \Bigg( \frac{xy}{(1+\alpha \xi)(\omega x^2 + 1) + x}\Bigg) \\
\frac{\mathrm{d} y}{\mathrm{d} t} & = e \Bigg( \frac{x + \xi (\omega x^2 + 1)}{(\alpha \xi+ 1)(\omega x^2 + 1) + x} \Bigg) y - m_1 y - m_2 y^2
\end{split}
\end{equation}

In real world scenarios, environmental fluctuations affect the dynamics of the system. In order to capture these fluctuations, we introduce the multiplicative white noise terms into (\ref{det3}). As in \cite{Stochastic1,Stochastic6, stochastic8}, we now suppose that the intrinsic growth rate of prey and the death rate of predator are mainly affected by environmental noise such that 

$$r \rightarrow r + \sigma_1 dW_1(t),\  m_1 \rightarrow m_1 +\sigma_2 dW_2(t)$$

where $W_i(t) \ (i = 1, 2)$ are the mutually independent standard Brownian motions with $ W_i(0) = 0$ and  $\sigma_1$ and $\sigma_2$ are positive constants and they represent the intensities of the white noise.

Also, the system can go through huge, occasionally catastrophic disturbances. Since white noise is a continuous noise, it cannot capture sudden environmental changes. To cater to these, we also apply a discontinuous stochastic process as L\'evy jumps to model these abrupt natural phenomenon as in \cite{la2010dynamics, LevyJumps2}. 

We now perturb $r$ and $m_1$ with discontinuous L\'evy noise in addition to the continuous white noise. So, we have

\begin{equation} \label{noise}
    r \rightarrow r + \sigma_1 \mathrm{d}W_1(t) + \int_{\mathbb{Y}}^{} \gamma_1 (v) \, \widetilde{N} (\mathrm{d}t,\mathrm{d}v) ,\  -m_1 \rightarrow - m_1 + \sigma_2 \mathrm{d}W_2(t) + \int_{\mathbb{Y}}^{} \gamma_2 (v) \, \widetilde{N} (\mathrm{d}t,\mathrm{d}v)
\end{equation}

According to the L\'evy decomposition theorem \cite{oksendal2005stochastic}, we have
$\widetilde{N} (t,dv) = N(t,dv) - \lambda (dv)t$, where $\widetilde{N}(t,dv)$ is a compensated Poisson process and N is a Poisson counting measure with characteristic measure $\lambda$ on a measurable subset $\mathbb{Y}$ of $(0,+\infty)$ with $\lambda(\mathbb{Y}) < \infty$. The distribution of L\'evy jumps L$_i$(t) can be completely parameterized by $(a_i,\sigma_i,\lambda)$ and satisfies the property of infinite divisibility.

Now, by incorporating noise induced parameters (\ref{noise}) into the reduced deterministic system of equations (\ref{det3}), we get the following additional food provided stochastic prey-predator system exhibiting Holling type-IV functional response along with the environmental fluctuations captured using the white noise and L\'evy noise.

\begin{equation} \label{stoc}
\begin{split}
\mathrm{d} x(t) & = x(t) \Bigg[r \Bigg(1-\frac{x(t)}{\gamma} \Bigg)- \Bigg( \frac{y(t)}{(1+\alpha \xi)(\omega x^2(t) + 1) + x(t)}\Bigg)\Bigg] \mathrm{d}t + \sigma_1 x(t) \mathrm{d}W_1(t) + x(t) \int_{\mathbb{Y}}^{} \gamma_1 (v) \, \widetilde{N} (\mathrm{d}t,\mathrm{d}v) \\
\mathrm{d} y(t) & = y(t) \Bigg[e \Bigg( \frac{x(t) + \xi (\omega x^2(t) + 1)}{(1 + \alpha \xi)(\omega x^2(t) + 1) + x(t)} \Bigg) - m_1 - m_2 y(t) \Bigg] + \sigma_2 y(t) \mathrm{d}W_2(t) + y(t) \int_{\mathbb{Y}}^{} \gamma_2 (v) \, \widetilde{N} (\mathrm{d}t,\mathrm{d}v)
\end{split}
\end{equation}

\section{Existence of Global Positive Solution} \label{sec3}

In order to study the stochastic dynamics of (\ref{stoc}), we first prove that the system (\ref{stoc}) has a unique global positive solution.

\begin{theorem} \label{th1}
For any given initial value $X(\theta) = (x(\theta),y(\theta)) \in C([-\tau_0,0],\mathbb{R}^{+^2})$, there exists a unique positive global solution $(x(t),y(t))$ of system (\ref{stoc}) on $t \geq 0$.
\end{theorem}

The above theorem for existence of solutions of (\ref{stoc}) can be proved in similar lines to the proof in \cite{LevyJumps2} using the Lyapunov method.

\section{Stochastic Time-Optimal Control Problems} \label{secOC}

In this section, we formulate and study the stochastic time-optimal control problems for prey-predator systems involving Holling type-IV functional response where the predator is provided with additional food. 

\subsection{Quality of additional food as control parameter}

In this subsection, we characterize the optimal quality of additional food for driving the system (\ref{stoc}) to a desired equilibrium state in minimum time using the stochastic maximum principle. So, we fix the quantity of additional food $\xi > 0$ to be a constant and choose the objective functional to be minimized for this stochastic time optimal control problem as follows.

\begin{equation}
 J(\alpha) = E \Bigg[ \int_{0}^{T} 1 dt \Bigg].  \label{obj}
\end{equation}  

From the Sufficient Stochastic Maximum Principle \cite{framstad2004sufficient} \nocite{oksendal2005stochastic} for the optimal control problems of jump diffusion, we characterize the optimal solution of the stochastic time optimal control problem with state space as the solutions of (\ref{stoc}) and the objective functional (\ref{obj}).

Let $(p^*,q^*,r^*)$ be a solution of the adjoint equation in the unknown processes $p(t) \in \mathbb{R}^2,\  q(t) \in \mathbb{R}^{2 \times 2},\  r(t,z) \in \mathbb{R}^2$ satisfying the backward differential equations 

\begin{equation} \label{adj}
\begin{split}
dp_1(t) &= \Bigg[ \left(-r+\frac{2rx}{\gamma} - \frac{2\omega(1+\alpha\xi)x + 1}{((1+\alpha \xi)(\omega x^2+1)+x)^2} \right) p_1(t) - \frac{(1-\omega x^2)(1+(\alpha-1)\xi)}{((1+\alpha \xi)(\omega x^2+1)+x)^2} e y p_2(t) - \sigma_1 q_1 \\&  - \int \gamma_1(v) r_1 v_1(dz_1) \Bigg] dt + q_1(t) dW_1(t) + q_2(t) dW_2(t) + \int_{\mathbb{R}^n}^{ } r_1 \widetilde{N}(dt,dz) \\
dp_2(t) &= - \Bigg[ \frac{-x}{(1+\alpha \xi)(\omega x^2 + 1)+x} p_1(t) + \left( \frac{e(x+\xi (\omega x^2 + 1))}{(1+\alpha \xi)(\omega x^2 + 1)+x}-m_1-2 m_2 y \right) p_2(t) + \sigma_2 q_4 \\& + \int \gamma_2 (v) r_2 v_2(dz_2) \Bigg] dt + q_3(t) dW_1(t) + q_4(t) dW_2(t) + \int_{\mathbb{R}^n}^{ } r_2 \widetilde{N}(dt,dz) \\
p_1(T) & = 0,  \ p_2(T) = 0
\end{split}
\end{equation}

The Hamiltonian associated with this control problem is defined as follows.

\begin{equation}
    \begin{split}
        H(t,x,y,\alpha,p,q,r) &= 1 + \Bigg[ r(1-\frac{x}{\gamma})-\frac{y}{(1+\alpha\xi)(1+\omega x^2)+x} \Bigg] xp_1 + \Bigg[\frac{e(x+\xi (\omega x^2+1))}{(1+\alpha \xi)(1+\omega x^2)+x} \\& -m_1-m_2 y\Bigg] y p_2 + \sigma_1 x q_1 + \sigma_2 y q_4 + x \int \gamma_1 r_1 v_1(dz_1) + y \int \gamma_2(v) r_2 v_2(dz_2)
    \end{split}
\end{equation}

Let $U = \{\alpha(t) | 0 \leq \alpha(t) \leq \alpha_{max} \forall t \in (0,t_f]\}$ where $\alpha_{max} \in \mathbb{R}^{+}.$ Let $\alpha^* \in U$ with the corresponding solution $(x^*,y^*) = (x(u^*),y(u^*))$. 

From the Arrow condition in the sufficient stochastic maximum Principle \cite{framstad2004sufficient}, we have 

\begin{equation*}
\begin{split}
    \frac{\partial H}{\partial \alpha} \Big|_{\alpha^*} = 0 & \\
    \implies & \Bigg[ - y x p_1 \frac{-\xi(\omega x^2 +1}{((1+\alpha \xi)(\omega x^2 + 1) + x)^2} - \frac{eyp_2(x+\xi(\omega x^2 + 1))(\xi (\omega x^2 + 1))}{((1+\alpha \xi)(\omega x^2 + 1) + x)^2}  \Bigg] \Bigg|_{*} = 0 \\
    \implies  & \Bigg[ yxp_1 - e y p_2 (x+\xi(\omega x^2 + 1)) \Bigg] \Bigg|_{*} = 0 \\
    \implies & x^* p_1^* = e p_2^* (x^* + \xi (\omega x^{*^2} + 1)) \\
\end{split}    
\end{equation*}

Hence the optimal control $\alpha^*$ should satisfy the following condition. 

\begin{equation} \label{u1}
    x^* p_1^* = e p_2^* (x^* + \xi (\omega x^{*^2} + 1))
\end{equation}
Since the analytical solution of (\ref{adj}) is complex to solve, we numerically simulate these results in section \ref{secNS}.

\subsection{Quantity of additional food as control parameter}

In this subsection, we characterize the optimal quantity of additional food for driving the system (\ref{stoc}) to a desired equilibrium state in minimum time using the stochastic maximum principle. So, we fix the quality of additional food $\alpha > 0$ to be a constant and choose the objective functional to be minimized for this stochastic time optimal control problem as follows.

\begin{equation}
 J(\xi) = E \Bigg[ \int_{0}^{T} 1 dt \Bigg].  \label{obj2}
\end{equation}

From the Sufficient Stochastic Maximum Principle \cite{framstad2004sufficient} \nocite{oksendal2005stochastic} for the optimal control problems of jump diffusion, we characterize the optimal solution of the stochastic time optimal control problem with state space as the solutions of (\ref{stoc}) and the objective functional (\ref{obj2}).

Let $(p^*,q^*,r^*)$ be a solution of the adjoint equation in the unknown processes $p(t) \in \mathbb{R}^2,\  q(t) \in \mathbb{R}^{2 \times 2},\  r(t,z) \in \mathbb{R}^2$ satisfying the backward differential equations 

\begin{equation} \label{adj2}
\begin{split}
dp_1(t) &= \Bigg[ \Big(-r+\frac{2rx}{\gamma} - \frac{2\omega(1+\alpha\xi)x + 1}{((1+\alpha \xi)(\omega x^2+1)+x)^2} \Big) p_1(t) - \frac{(1-\omega x^2)(1+(\alpha-1)\xi)}{((1+\alpha \xi)(\omega x^2+1)+x)^2} e y p_2(t) - \sigma_1 q_1 \\&  - \int \gamma_1(v) r_1 v_1(dz_1) \Bigg] dt + q_1(t) dW_1(t) + q_2(t) dW_2(t) + \int_{\mathbb{R}^n}^{ } r_1 \widetilde{N}(dt,dz) \\
dp_2(t) &= - \Bigg[ \frac{-x}{(1+\alpha \xi)(\omega x^2 + 1)+x} p_1(t) + \Big(\frac{e(x+\xi (\omega x^2 + 1))}{(1+\alpha \xi)(\omega x^2 + 1)+x}-m_1-2 m_2 y \Big)p_2(t) + \sigma_2 q_4 \\& + \int \gamma_2 (v) r_2 v_2(dz_2) \Bigg] dt + q_3(t) dW_1(t) + q_4(t) dW_2(t) + \int_{\mathbb{R}^n}^{ } r_2 \widetilde{N}(dt,dz) \\
p_1(T) & = 0,  \ p_2(T) = 0
\end{split}
\end{equation}

The Hamiltonian associated with this control problem is defined as follows.
\begin{equation}
    \begin{split}
        H(t,x,y,\xi,p,q,r) &= 1 + \Bigg[ r(1-\frac{x}{\gamma})-\frac{y}{(1+\alpha\xi)(1+\omega x^2)+x} \Bigg] xp_1 + \Bigg[\frac{e(x+\xi (\omega x^2+1))}{(1+\alpha \xi)(1+\omega x^2)+x} \\& -m_1-m_2 y\Bigg] y p_2 + \sigma_1 x q_1 + \sigma_2 y q_4 + x \int \gamma_1 r_1 v_1(dz_1) + y \int \gamma_2(v) r_2 v_2(dz_2)
    \end{split}
\end{equation}

Let $U = \{\xi(t) | 0 \leq \xi(t) \leq \xi_{max} \forall t \in (0,t_f]\}$ where $\xi_{max} \in \mathbb{R}^{+}.$ Let $\xi^* \in U$ with the corresponding solution $(x^*,y^*) = (x(\xi^*),y(\xi^*))$. 

From the Arrow condition in the sufficient stochastic maximum Principle \cite{framstad2004sufficient}, we have 
\begin{equation*}
\begin{split}
    \frac{\partial H}{\partial \xi} \Big|_{\xi^*} = 0 & \\
    \implies & \Bigg[ - y x p_1 \frac{-\alpha(\omega x^2 +1}{((1+\alpha \xi)(\omega x^2 + 1) + x)^2} + \\ & e y p_2 \left( \frac{((1+\alpha \xi) (\omega x^2 + 1) + x) (\omega x^2 + 1) - \alpha (\omega x^2 + 1) (x+\xi (\omega x^2 + 1))}{((1+\alpha \xi)(\omega x^2 + 1) + x)^2} \right) \Bigg] \Bigg|_{*} = 0 \\
    \implies  & \Bigg[ \alpha x y (\omega x^2 + 1) p_1 + e y p_2 (\omega x^2 + 1) (x(1-\alpha) (\omega x^2 + 1)) \Bigg] \Bigg|_{*} = 0 \\
    \implies & \Bigg[ \alpha x p_1 + e p_2 (1 + \omega x^2 + x (1 -\alpha))\Bigg] \Bigg|_{*} = 0 \\ 
    \implies & \alpha x^* p_1^* + e p_2^* (1 + \omega x^{*^2} + x^* (1-\alpha)) = 0 \\
\end{split}    
\end{equation*}

Hence the optimal control $\xi^*$ should satisfy the following condition. 

\begin{equation} \label{u2}
    \alpha x^* p_1^* + e p_2^* (1 + \omega x^{*^2} + x^* (1-\alpha)) = 0
\end{equation}

Since the analytical solution of (\ref{adj}) is complex to solve, we numerically simulate these results in section \ref{secNS}.

\subsection{Existence and Uniqueness of Solutions for the Forward Backward Stochastic Differential Equations with Jumps (FBSDEJ)}

We so far obtained the adjoint equations (\ref{adj}),(\ref{adj2}) for the state equations (\ref{stoc}) and the objective functional (\ref{obj}), (\ref{obj2}) using the sufficient stochastic maximum principle respectively. Upon simplifying the results obtained from the arrow condition (\ref{u1}), (\ref{u2}) from earlier two subsections, we see that the optimal controls are given by

\begin{equation} \label{u}
        \alpha^* = \frac{ep_2^* (1+x^*+\omega x^{*^2})}{e p_2^* x^*-p_1^* x^*}, \  \xi^* = \frac{x^* p_1^* - e x^* p_2^*}{e p_2^* (1+\omega x^{*^2})}
\end{equation}

In this section, we now prove the existence of optimal controls by proving the existence of the solutions for the FBSDEJ ((\ref{stoc}),(\ref{adj}),(\ref{adj2})) which establishes the existence of $(x^*,y^*,p_1^*,p_2^*)$ for all simulation purposes. Using the theorem in \cite{JumpExistence}, we now prove the existence of the optimal controls (\ref{u}) in the following theorem.
\begin{theorem} \label{th2}
For any $(x_0,y_0) \in \mathbb{R}^{+^2}$, the FBSDEJ ((\ref{stoc}),(\ref{adj}),(\ref{adj2})) admits an optimal stochastic control.
\end{theorem}

\begin{proof}
Let $(X_t)_{t \geq 0}$ be the solution of the Stochastic Differential Equation with Jumps(SDEJ) 
\begin{equation*}
    \mathrm{d}X_t = b(X_t) \mathrm{d}t + \sigma(X_t) \mathrm{d}W(t) + \int_{\mathbb{R}}^{ } \Gamma(v) \widetilde{N} (\mathrm{d}t, \mathrm{d}v)
\end{equation*}

Here the term $b(X_t)$ denotes the drift coefficient, the term $\sigma(X_t)$ denotes the diffusion coefficient and the term $\Gamma(v)$ denotes the poisson term coefficient.

The theorem \ref{th1} in section \ref{sec3} guaranties the monotonicity and Lipschitz continuity of the drift coefficient, the diffusion coefficient and the poisson term coefficient of the state equations (\ref{stoc}).

Following the the existence and uniqueness theorem of FBDSDEJ in \cite{JumpExistence}, we are only left to prove the monotonicity and Lipschitz continuity of the drift and diffusion terms of the adjoint system of equations (\ref{adj}). 

From (\ref{adj}), due to the positivity of state variables guaranteed by theorem \ref{th1}, the drift term and the diffusion terms are given as follows.
\begin{equation*}
    b(X_t) \leq \begin{pmatrix}  \left( \frac{2rx}{\gamma} \right) p_1(t) + \left( \omega e x^2 y  \right) p_2(t) \\ (x) p_1(t) + (m_1 + 2 m_2 y) p_2(t) \end{pmatrix}, \ \sigma(X_t) = \begin{pmatrix}   q_1(t) & q_2(t) \\ q_3(t) & q_4(t) \end{pmatrix}
\end{equation*}

Since the drift coefficient is a linear combination of adjoint terms ($p_1, \ p_2$), the monotonicity and Lipschitz continuity are guaranteed. 

In addition to this, the diffusion coefficient is independent of the adjoint terms $(p_1, \ p_2)$. Therefore, the monotonicity and Lipschitz continuity are guarenteed for the diffusion coefficients.

Hence the existence of unique stochastic optimal controls are proved for FBSDEJ ((\ref{stoc}),(\ref{adj}),(\ref{adj2})).
\end{proof}

\section{Numerical Simulations} \label{secNS}

In this section, we perform the extensive numerical simulations using python by choosing the following parameters \cite{LevyJumps2} for the model (\ref{stoc}). $r=1.5,\ \gamma = 12, \omega = 15, \ e =0.4, m_1 = 0.15, \ m_2 = 0.01,\  \sigma_1 = \sigma_2 = 0.02, \ \gamma_1 = 1, \gamma_2 = 1$. In these simulations, white noise is simulated using the Box-Normal transformations and the poisson noise is simulated using the poisson point processes \cite{NSJumps}. The state equations (\ref{stoc}) and the adjoint equations (\ref{adj}),(\ref{adj2}) are simulated using the Forward Backward Doubly Stochastic Differential Equations with Jumps (FBDSDEJ) method. The subplots in figures \ref{fig1}, \ref{fig2} depict the optimal state trajectories, optimal co-state trajectories, phase diagram, optimal quality of additional food and the optimal quantity of additional food respectively.
\paragraph{\qquad \underline{A. Applications to Biological Conservation:}}

The subplot (\ref{11}) depicts the optimal state trajectory of the system (\ref{stoc}) from the initial state $(2,8)$ that stabilizes over time around the state $(16,90)$. The subplot (\ref{13}) gives the phase diagram which shows the trajectories are stabilized over high values of prey and predator. The subplots (\ref{14}) and (\ref{15}) depicts the optimal quality and quantity of additional food respectively. These plots show that the high quality of additional food is required to achieve biological conservation. Even if the quantity of additional food is lower, still we will be able to achieve biological conservation with higher quality of additional food.
\begin{figure} 
    \begin{subfigure}{0.45\textwidth}
        \includegraphics[width=\textwidth]{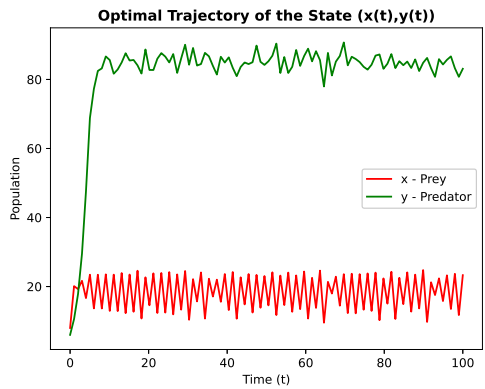}
        \caption{ }
        \label{11}
    \end{subfigure}
    \begin{subfigure}{0.45\textwidth}
        \includegraphics[width=\textwidth]{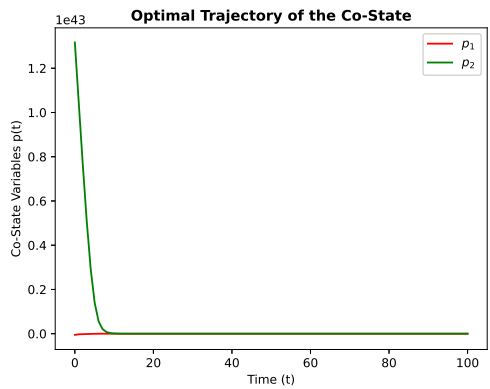}
        \caption{ }
        \label{12}
    \end{subfigure}

    \begin{center}        
    \begin{subfigure}{0.45\textwidth}
        \includegraphics[width=\textwidth]{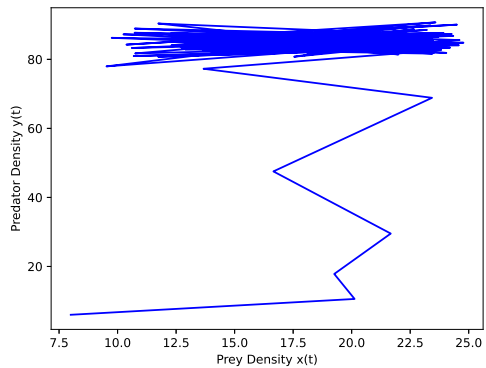}
        \caption{ }
        \label{13}
    \end{subfigure}
    \end{center}

     \begin{subfigure}{0.45\textwidth}
        \includegraphics[width=\textwidth]{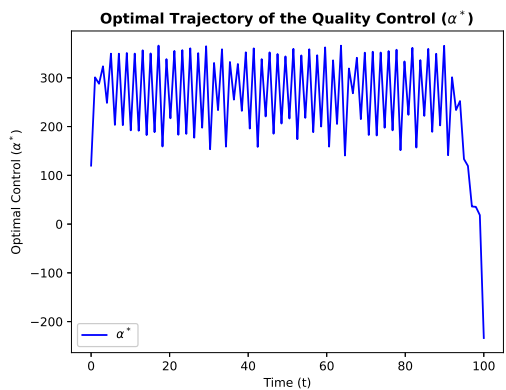}
        \caption{ }
        \label{14}
    \end{subfigure}
    \begin{subfigure}{0.45\textwidth}
        \includegraphics[width=\textwidth]{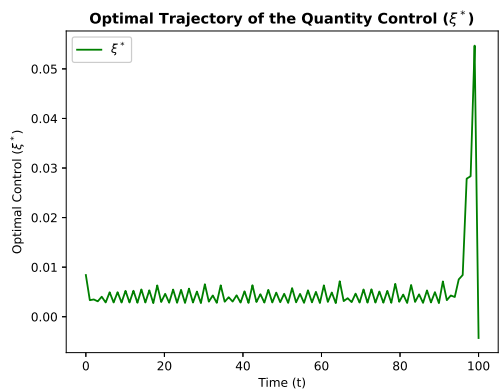}
        \caption{ }
        \label{15}
    \end{subfigure}
    \caption{This ﬁgure depicts the simulations of time-optimal control problem with respect to the control parameters in the context of biological conservation.}
    \label{fig1}
\end{figure}

\paragraph{\qquad \underline{B. Applications to Pest Management:}}
The subplot (\ref{21}) depicts the optimal state trajectory of the system (\ref{stoc}) from the initial state $(2,8)$ that reaches the nearly prey-elimination stage around the state $(5,90)$. The subplot (\ref{23}) depicts this property more clearly through the phase diagram where it reaches the lowest prey value over the time. The subplots (\ref{24}) and (\ref{25}) depicts that a lesser quality of additional food and a lower quantity of additional food is good enough to achieve pest management where pest is viewed as prey.

\begin{figure} 
    \begin{subfigure}{0.45\textwidth}
        \includegraphics[width=\textwidth]{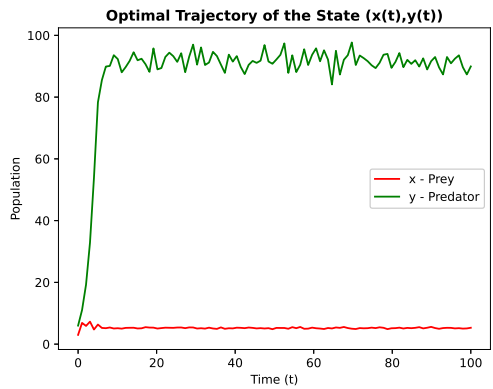}
        \caption{ }
        \label{21}
    \end{subfigure}
    \begin{subfigure}{0.45\textwidth}
        \includegraphics[width=\textwidth]{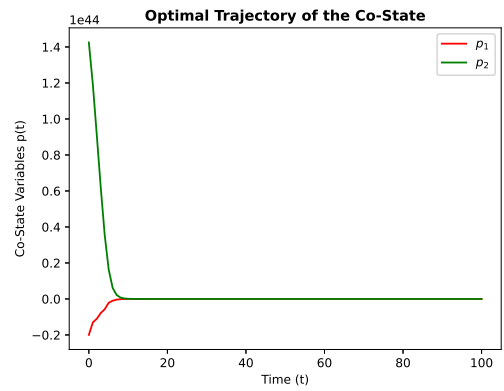}
        \caption{ }
        \label{22}
    \end{subfigure}

    \begin{center}        
    \begin{subfigure}{0.45\textwidth}
        \includegraphics[width=\textwidth]{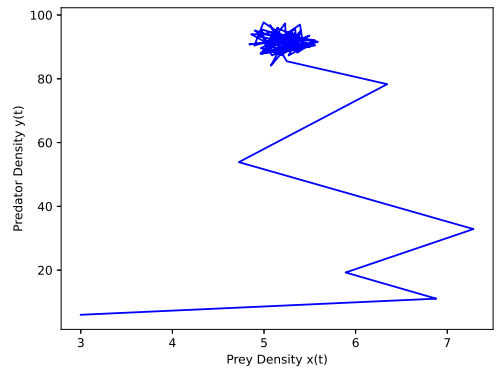}
        \caption{ }
        \label{23}
    \end{subfigure}
    \end{center}

     \begin{subfigure}{0.45\textwidth}
        \includegraphics[width=\textwidth]{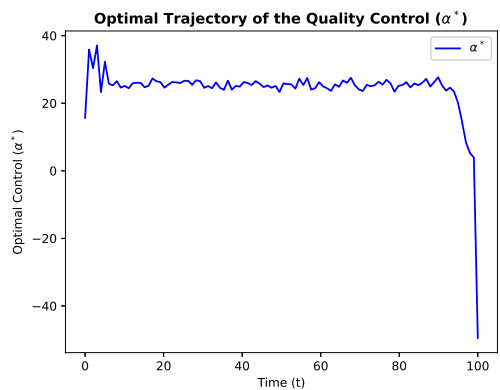}
        \caption{ }
        \label{24}
    \end{subfigure}
    \begin{subfigure}{0.45\textwidth}
        \includegraphics[width=\textwidth]{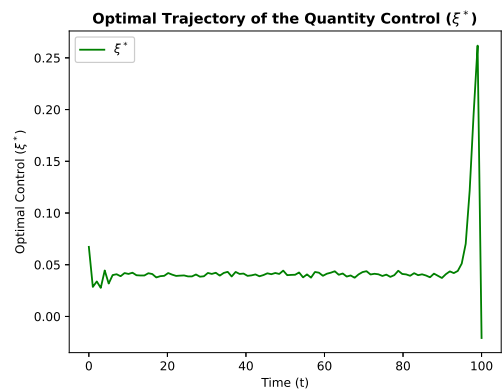}
        \caption{ }
        \label{25}
    \end{subfigure}
    \caption{This ﬁgure depicts the simulations of time-optimal control problem with respect to the control parameters in the context of pest management.}
    \label{fig2}
\end{figure}

\newpage

\section{Conclusion} \label{secD}

This paper is about a stochastic prey-predator system exhibiting Holling type IV functional response along with the combined white noise and L\'evy noise. We do the time-optimal control studies for the system, where the quality of additional food and the quantity of additional food are considered as control parameters respectively. To begin with, we extend the model by incorporating multiplicative noise to both prey and predator. In theorem \ref{th1}, we proved the existence of a unique positive global solution of (\ref{stoc}). Further, we formulated the time-optimal control problem with the objective to minimize the final time in which the system reaches the pre-defined state. Using the sufficient stochastic maximum principle, we characterized the optimal control values. In theorem \ref{th2}, we proved that the existence and uniqueness of Forward Backward Doubly Stochastic Differential Equations with Jumps (FBDSDEJ). Lastly, we numerically simulated the theoretical findings and applied them in the context of biological conservation and pest management. 

Some of the salient features of this work include the following. Unlike the most traditional papers, here we considered a stochastic time-optimal control problem. As Intra-specific competition among predators is ineluctable, we also explicitly incorporated the intra-specific competition into our model. This paper mainly deals with the novel study of the time-optimal control problems where the state equations involve both the discrete and continuous noise which is challenging. In future research, we hope to add and study more specificalities such as mutual interference, stochastic bifurcations, Markov chain and partial differential systems. 

\subsection*{Financial Support: }
This research was supported by National Board of Higher Mathematics(NBHM), Government of India(GoI) under project grant - {\bf{Time Optimal Control and Bifurcation Analysis of Coupled Nonlinear Dynamical Systems with Applications to Pest Management, \\ Sanction number: (02011/11/2021NBHM(R.P)/R$\&$D II/10074).}}

\subsection*{Conflict of Interests Statement: }
The authors have no conflicts of interest to disclose.

\subsection*{Ethics Statement:} 
This research did not required ethical approval.

\subsection*{Acknowledgments}
The authors dedicate this paper to the founder chancellor of SSSIHL, Bhagawan Sri Sathya Sai Baba. The corresponding author also dedicates this paper to his loving elder brother D. A. C. Prakash who still lives in his heart.

\printbibliography

\end{document}